\documentclass[reqno]{amsart}

\usepackage[utf8]{inputenc}
\usepackage{amssymb,amsmath,amsthm}
\usepackage{mathtools}
\usepackage{cleveref}

\numberwithin{equation}{section}

\theoremstyle{plain}
\newtheorem{theorem}{Theorem}[section]
\newtheorem{lemma}[theorem]{Lemma}

\newtheorem{corollary}[theorem]{Corollary}

\theoremstyle{definition}
\newtheorem{definition}[theorem]{Definition}

\theoremstyle{definition}
\newtheorem{example}[theorem]{Example}

\DeclareSymbolFont{bbold}{U}{bbold}{m}{n}
\DeclareSymbolFontAlphabet{\mathbbold}{bbold}

\newcommand{\N}{\mathbb{N}}
\newcommand{\C}{\mathbb{C}}
\newcommand{\h}{\mathfrak{H}}
\newcommand{\cl}{c\ell}
\newcommand{\D}{\mathfrak{D}}
\newcommand{\comp}{{^\mathrm{c}}}
\newcommand{\heyting}{\mathbin{\rightarrow}}

\let\olddownarrow\downarrow
\renewcommand{\downarrow}{\mathord{\olddownarrow}}
\let\olduparrow\uparrow
\renewcommand{\uparrow}{\mathord{\olduparrow}}

\begin{document}

\title[A special class of congruences on $\kappa$-frames]{A special class of congruences on $\kappa$-frames}

\author[G. Manuell]{Graham Manuell}
\email{graham@manuell.me}
\address{Department of Mathematics and Applied Mathematics, University of Cape Town, Private Bag Rondebosch, Cape Town 7701, South Africa}

\thanks{I acknowledge financial assistance from the National Research Foundation of South Africa.}


\begin{abstract}
 Madden has shown that in contrast to the situation with frames, the smallest dense quotient of a $\kappa$-frame need not be Boolean.
 We characterise these so-called \textit{d}-reduced $\kappa$-frames as those which may be embedded as a generating sub-$\kappa$-frame
 of a Boolean frame. We introduce the notion of the closure of a $\kappa$-frame congruence and call a congruence \emph{clear} if it is
 the largest congruence with a given closure. These ideas are used to prove $\kappa$-frame analogues of known results concerning Boolean frame quotients.
 In particular, we show that \textit{d}-reduced $\kappa$-frames are precisely the quotients of $\kappa$-frames by clear congruences and that every
 $\kappa$-frame congruence is the meet of clear congruences.
\end{abstract}

\maketitle

\section{Background}\label{section:background}

For background on frames and $\kappa$-frames see \cite{PicadoPultr} and \cite{Madden}.
A frame is a complete lattice satisfying the distributivity law
$x \wedge \bigvee_{\alpha \in I} x_\alpha = \bigvee_{\alpha \in I} (x \wedge x_\alpha)$ for arbitrary families $(x_\alpha)_{\alpha \in I}$.
We denote the smallest element of a frame by $0$ and the largest element by $1$. A frame homomorphism is a function which preserves finite
meets and arbitrary joins.

We use $\kappa$ to denote a fixed regular cardinal. A $\kappa$-set is then a set of cardinality strictly less than $\kappa$.
A $\kappa$-frame is a bounded distributive lattice which has joins of $\kappa$-sets and satisfies the frame distributivity law
when $|I| < \kappa$. Morphisms of $\kappa$-frames are functions which preserve finite meets and $\kappa$-joins (i.e.\ joins of $\kappa$-sets).
For $\kappa = \aleph_1$ we obtain $\sigma$-frames as a special case and for $\kappa = \aleph_0$ we recover bounded distributive lattices.
Results we prove about $\kappa$-frames will also hold for frames after making the obvious changes.

Recall that a poset $P$ is said to be \emph{$\kappa$-directed} if every $\kappa$-subset $S \subseteq P$ has an upper bound in $P$.
A \emph{$\kappa$-ideal} on a $\kappa$-frame is a $\kappa$-directed downset. The set of $\kappa$-ideals $\h_\kappa L$ on a $\kappa$-frame $L$ yields a frame
when ordered by inclusion. This is the free frame generated by $L$.

A congruence on a $\kappa$-frame $L$ is an equivalence relation on $L$ that is also a sub-$\kappa$-frame of $L \times L$.
The following lemma is well-known.

\begin{lemma}\label{lem:congruence_intervals}
 Let $L$ be a $\kappa$-frame and let $C$ be a congruence on $L$. Then $(a,b) \in C$ if and only if there are $x,y \in L$ with $x \le a,b \le y$ such
 that $(x,y) \in C$.
\end{lemma}

The congruence generated by $(0,a)$ for some $a \in L$ is called \emph{principal closed} and denoted by $\nabla_a$.
The congruence generated by $(a,1)$ is called \emph{open} and denoted by $\Delta_a$. We can explicitly compute
$\nabla_a = \{(x,y) \mid x \vee a = y \vee a\}$ and $\Delta_a = \{(x,y) \mid x \wedge a = y \wedge a\}$.
The join of one of these congruences with an arbitrary congruence $C$ can similarly be found to give
$\nabla_a \vee C = \{(x,y) \mid (x \vee a, y \vee a) \in C\}$ and $\Delta_a \vee C = \{(x,y) \mid (x \wedge a, y \wedge a) \in C\}$.

The lattice $\C L$ of all congruences on a $\kappa$-frame $L$ is a frame and the assignment $a \mapsto \nabla_a$ is an injective $\kappa$-frame homomorphism.
The congruences $\nabla_a$ and $\Delta_a$ are complements of each other and the principal closed and open congruences together generate $\C L$.

\section{Closure}

For frames, principal closed congruences are closed under arbitrary joins, but in the case of $\kappa$-frames, they are only closed
under joins of cardinality less than $\kappa$. Nevertheless, it is useful to consider arbitrary joins of principal closed congruences.

\begin{definition}
 A congruence on a $\kappa$-frame is called \emph{closed} if it is a join of principal closed congruences.
 If $I$ is a $\kappa$-ideal of a $\kappa$-frame, we define the closed congruence $\nabla_I = \bigvee_{a \in I} \nabla_a$.
\end{definition}
We can use closure under $\kappa$-joins to reduce any join of principal closed congruences to a $\kappa$-directed join. Also note that if a
principal closed congruence $\nabla_a$ is included in such a join, we might as well include any smaller principal closed congruences.
But a $\kappa$-directed lower set is nothing but a $\kappa$-ideal. So every closed congruence is of the form $\nabla_I$ for some $\kappa$-ideal $I$ on $L$.

We can compute joins with general closed congruences in a similar manner to as with principal closed congruences.
\begin{lemma}\label{lem:joins_with_generalised_closed}
 If $L$ is a $\kappa$-frame, $C \in \C L$ and $I$ is a $\kappa$-ideal on $L$, then
 \[\nabla_I \vee C = \{(x,y) \mid (x \vee i, y \vee i) \in C \text{ for some $i \in I$}\}.\]
 In particular, $\nabla_I = \{(x,y) \mid x \vee i = y \vee i \text{ for some $i \in I$}\}$.
\end{lemma}

\begin{lemma}\label{lem:nabla_tilde_injective}
 The map $\nabla_\bullet \colon \h_\kappa L \to \C L$ is an injective frame homomorphism.
\end{lemma}

The composition of $\nabla_\bullet$ with its right adjoint gives a map $\cl = \nabla_\bullet \circ (\nabla_\bullet)_*$,
which is of some interest. The map $\cl \colon \C L \to \C L$ assigns to a congruence the largest closed congruence lying below it.
More explicitly we have $\cl(C) = \nabla_{[0]_C}$ where the $\kappa$-ideal $[0]_C$ is the equivalence class of $0$ in $C$.
The map $\cl$ is an analogue of the topological closure operation and, remembering that the lattice of congruences is ordered in the reverse order
to the lattice of sublocales (the lattice of regular quotients \cite{PicadoPultr}), we expect it to be monotone, deflationary and idempotent and to
preserve finite meets. This is easily seen to be the case.
\begin{definition}
 If $C$ is a congruence, $\cl(C)$ is called the \emph{closure} of $C$.
\end{definition}

We recall the correspondence between quotients of $L/C$ and quotients of $L$ by congruences lying above $C$.
The canonical isomorphism $(L/C)/\nabla_{[a]} \cong L / (C \vee \nabla_a)$ implies that this correspondence respects (principal) closed congruences.
This results in the following lemma.
\begin{lemma}\label{lem:closure_in_quotient}
 Suppose $A \ge C \in \C L$. The closure of the image of $A$ in $\C(L/C)$ corresponds to $C \vee \cl(A)$ in $\C L$.
\end{lemma}

\section{Clear congruences}\label{subsec:dense_congruences}

\begin{definition}
 Suppose $C,D \in \C L$ and $C \le D$. We say $D$ is \emph{dense in $C$} if $\cl(D) \le C$.
 We say $D$ is \emph{dense} if it is dense in $0$.
\end{definition}
This definition is quickly seen to be compatible with the notion of dense maps.
\begin{lemma}\label{lem:dense}
 A congruence $D \in \C L$ is dense in $C$ (where $C \le D$) if and only if the canonical map $h\colon L/C \to L/D$ is dense.
\end{lemma}
\begin{proof}
 Let $q_A\colon L \twoheadrightarrow L/A$ denote the quotient map associated with a congruence $A$.
 Now $h$ is dense if and only if $q_D(a) = 0 \implies q_C(a) = 0$.
 But $q_A(a) = 0 \iff \nabla_a \le A$ and so $h$ is dense
 if and only if $\nabla_a \le D \implies \nabla_a \le C$, which is to say that $\cl(D) \le C$.
\end{proof}

Madden shows in \cite{Madden} that a $\kappa$-frame $L$ has a largest dense congruence given by
\[\D_L = \{(a,b) \mid a \wedge x = 0 \iff b \wedge x = 0 \,\text{ for all $x \in L$}\}.\]
Now by applying \cref{lem:closure_in_quotient} we obtain the following.

\begin{lemma}\label{lem:clear_cong_characterisation}
 For every $\kappa$-ideal $I$ on $L$, there is a largest congruence $\partial_I$ dense in $\nabla_I$
 corresponding to $\D_{L/\nabla_I}$ in $\C (L/\nabla_I)$.
 More explicitly we find \[\partial_I = \{(a,b) \mid a \wedge x \in I \iff b \wedge x \in I \,\text{ for all $x \in L$}\}.\]
\end{lemma}

\begin{corollary}\label{cor:clear_cong_heyting_arrow}
 $\partial_I = \{(a,b) \mid \downarrow a \heyting I = \downarrow b \heyting I\}$ where `$\heyting\!$' denotes the
 Heyting arrow in $\h_\kappa I$.
\end{corollary}

If $C$ is a congruence and $\nabla_I = \cl(C)$, then $\partial_I$ is the largest congruence dense in $C$ and we have $\nabla_I \le C \le \partial_I$.
We may compute this $\partial_I$ explicitly as \[\partial_I = \{(a,b) \mid a \wedge x \sim_C 0 \iff b \wedge x \sim_C 0 \,\text{ for all $x \in L$}\}.\]

The congruences of the form $\partial_I$ are an important class of congruences. In the case of frames, their associated nuclei (those of the form
$w_a(x) = (x \heyting a) \heyting a$) and sublocales (Boolean sublocales) have been studied in \cite{Macnab}, \cite{Isbell1991} and elsewhere,
but the congruence-oriented approach remains largely unexplored. There has consequently been little discussion in the setting of $\kappa$-frames where
nuclei and sublocales are unavailable. Some related ideas and results appear in \cite{Madden}, but the congruences $\partial_I$ narrowly escape mention.
Since we might imagine the smallest dense sublocale of a frame to be so rarefied as to appear translucent, we will call these \emph{clear} congruences.
\begin{definition}
 A \emph{clear} congruence is a congruence of the form $\partial_I$.
\end{definition}

The assignment $I \mapsto \partial_I$ is injective and reflects order, but the following example shows it is not monotone.
\begin{example}\label{ex:clear_congs_incomparable}
 The chain $\mathbbold{3} = \{0, a, 1\}$ is the frame of opens of the Sierpiński space.
 The clear congruences $\partial_0$ and $\partial_{\downarrow a}$ correspond to the open and closed points respectively,
 so $\partial_0$ and $\partial_{\downarrow a}$ are incomparable.
\end{example}

The clear congruences generate the congruence lattice under meets.
\begin{theorem}\label{thm:meet_of_clear}
 Every congruence is a meet of clear congruences.
\end{theorem}
\begin{proof}
 Take $C \in \C L$. Certainly, $C \le \bigwedge\{\partial_I \mid \partial_I \ge C,\, I \in \h_\kappa L\}$.
 Now suppose $(a,b) \notin C$. We may assume $a < b$ and we need a $\kappa$-ideal $I$ such that $\partial_I \ge C$ and $(a,b) \notin\partial_I$.
 
 Let $\nabla_I = \cl(\nabla_a \vee C)$. We have $\nabla_I \le \nabla_a \vee C \le \partial_I$ and thus $\partial_I \ge C$.
 Also, $\nabla_a \le \nabla_I$ and so $a \in I$. (Indeed, one may note that $I = \downarrow[a]_C$.)
 
 Suppose that $b \in I$ as well. Then $(a,b) \in \nabla_I \le \nabla_a \vee C$. So $(a,b) = (a\vee a, a\vee b) \in C$.
 This is a contradiction and so $b \notin I$. But then $(a,b) \notin \partial_I$ by \cref{lem:clear_cong_characterisation}.
\end{proof}

Quotients of a $\kappa$-frame by clear congruences have a special form.
\begin{definition}
 A $\kappa$-frame $L$ is called \emph{\textit{d}-reduced} \cite{Madden} if $\D_L = 0$ in $\C L$.
\end{definition}
\begin{lemma}\label{lem:quotient_by_clear}
 The quotient $L/C$ is \textit{d}-reduced if and only if the congruence $C$ is clear.
\end{lemma}
\begin{proof}
 By \cref{lem:closure_in_quotient}, $L/C$ is \textit{d}-reduced if and only if $C$ is the largest congruence $D$ for which $\cl(D) \vee C \le C$.
 But $\cl(D) \vee C \le C \iff \cl(D) \le C \iff \cl(D) \le \cl(C)$ and $C$ is the largest such $D$ if and only if $C$ is clear.
\end{proof}

\begin{lemma}\label{lem:clear_distinct_pseudocomplements}
 A $\kappa$-frame $L$ is \textit{d}-reduced if and only if no two distinct principal ideals on $L$ have the same pseudocomplement in $\h_\kappa L$.
\end{lemma}
\begin{proof}
 Simply note that $\partial_0 = \{(a,b) \mid (\downarrow a)^* = (\downarrow b)^*\}$ by \cref{cor:clear_cong_heyting_arrow}.
\end{proof}
As a corollary we obtain a known result about frames.
\begin{corollary}\label{cor:clear_frame_Boolean}
 A frame is \textit{d}-reduced if and only if it is Boolean.
\end{corollary}
\begin{proof}
 Suppose $L$ is a \textit{d}-reduced frame. The frame $\h_\infty L$ of principal ideals on $L$ is isomorphic to $L$ itself and so every element of $L$
 has a unique pseudocomplement by \cref{lem:clear_distinct_pseudocomplements}. But for $a \in L$, we have $(a \vee a^*)^* = a^* \wedge a^{**} = 0$ and so
 $a \vee a^* = 1$. Thus, $a$ is complemented and $L$ is Boolean.
 
 Conversely, if $L$ is Boolean, then every element of $L$ is complemented and so pseudocomplements are clearly unique.
\end{proof}

We are now in a position to prove the following characterisation of \textit{d}-reduced $\kappa$-frames.
\begin{theorem}\label{thm:clear_kappa_frames}
Let $L$ be a $\kappa$-frame. The following are equivalent:
\begin{enumerate}
 \item $L$ is \textit{d}-reduced
 \item The map $g\colon L \hookrightarrow \h_\kappa L \twoheadrightarrow \h_\kappa L/\D_{\h_\kappa L}$ sending $a$ to $[\downarrow a]$ is injective
 \item $L$ is embeds as a sub-$\kappa$-frame of a Boolean frame $B$ and this sub-$\kappa$-frame generates $B$ under arbitrary joins.
\end{enumerate}
\end{theorem}
\begin{proof}
 ($1 \Rightarrow 2$) Suppose $L$ is \textit{d}-reduced. Take $x,y \in L$. By \cref{cor:clear_cong_heyting_arrow},
 $[\downarrow x] = [\downarrow y]$ in $\h_\kappa L/\D_{\h_\kappa L}$ if and only if $(\downarrow x)^* = (\downarrow y)^*$.
 But by \cref{lem:clear_distinct_pseudocomplements}, $(\downarrow x)^*$ and $(\downarrow y)^*$ are distinct if $x$ and $y$ are.
 Thus $g\colon a \mapsto [\downarrow a]$ is injective.
 
 ($2 \Rightarrow 3$) The frame $\h_\kappa L/\D_{\h_\kappa L}$ is Boolean by \cref{cor:clear_frame_Boolean} and is generated by the image
 of $L$ under the map $g$ since $\h_\kappa L$ is generated by the image of $L$ and the quotient map is surjective.
 
 ($3 \Rightarrow 1$) We now assume $L$ is a sub-$\kappa$-frame of a complete Boolean algebra $B$ and that $L$ generates $B$ under arbitrary joins.
 Let $a \in L$ and consider its complement $a\comp \in B$. Suppose $b \in L$ such that $b \le a\comp$ in $B$. Then $b \wedge a = 0$ and so
 $b \in (\downarrow a)^* \in \h_\kappa L$. Conversely, if $b \in (\downarrow a)^*$ then $b \wedge a = 0$ and so $b \le a\comp$.
 Thus $(\downarrow a)^* =\; \downarrow a\comp \cap L$. Since $L$ generates $B$, $a\comp = \bigvee (\downarrow a)^*$ in $B$.
 So for any $a, b \in L$, $(\downarrow a)^* = (\downarrow b)^*$ only if $a\comp = b\comp$ only if $a = b$. Thus, $L$ is \textit{d}-reduced.
\end{proof}

\begin{example}\label{ex:clear_kappa_frame}
 Let $M$ be the lattice of subsets of $\N$ that are either finite or equal to $\N$.
 Since $M$ is a generating sublattice of $2^\N$, it is a \textit{d}-reduced distributive lattice by \cref{thm:clear_kappa_frames}.
 Now consider the ideal $I$ of $M$ consisting of the sets that do not contain the natural number $2$.
 This is the pseudocomplement of the principal ideal generated by $\{2\}$.
 The quotient lattice $M / \nabla_I$ consists of three elements, namely $[\emptyset]$, $[\{2\}]$ and $[\N]$, and is
 isomorphic to the 3-element frame $\mathbbold{3}$. Thus, quotients of \textit{d}-reduced $\kappa$-frames may fail to be \textit{d}-reduced.
 It would be interesting to know exactly which $\kappa$-frames can appear as quotients of \textit{d}-reduced ones, but we do not attempt to answer this here.
\end{example}

\begin{lemma}\label{lem:hereditary_clear_is_Boolean}
 Every quotient of a $\kappa$-frame $L$ is \textit{d}-reduced if and only if $L$ is Boolean.
\end{lemma}
\begin{proof}
 $(\Leftarrow)$ A quotient of a Boolean $\kappa$-frame is Boolean and Boolean $\kappa$-frames are \textit{d}-reduced.
 
 $(\Rightarrow)$
 Let $a \in L$. We will show that $a$ has a complement. Consider $\Delta_a \in \C L$ and
 let $\nabla_I = \cl(\Delta_a)$. But $\nabla_I$ is clear, so $\nabla_I = \partial_I$
 and thus $\nabla_I = \Delta_a$.
 
 Now $\nabla_a \vee \nabla_I = 1$ and so $(0,1) \in \nabla_a \vee \nabla_I$. So $(a,1) \in \nabla_I$
 and $a \vee i = 1$ for some $i \in I$. But then $\nabla_a \vee \nabla_i = 1$ and also $\nabla_a \wedge \nabla_i = 0$ since
 $\nabla_i \le \nabla_I$. Thus $i$ is the complement of $a$ in $L$ and $L$ is Boolean.
\end{proof}

While arbitrary closed quotients of \textit{d}-reduced $\kappa$-frames typically fail to be \textit{d}-reduced,
quotients by principal closed congruences always are.
\begin{lemma}\label{lem:principal_closed_in_clear_is_clear}
 If $L$ is \textit{d}-reduced, then $\nabla_a$ is clear for all $a \in L$.
\end{lemma}
\begin{proof}
 Embed $L$ as a generating sub-$\kappa$-frame of a Boolean frame $B$. Since $L/\nabla_a \cong \uparrow a \subseteq L$, this embedding induces an
 embedding of $L/\nabla_a$ as a generating sub-$\kappa$-frame of $B/\nabla_a$ and hence $L/\nabla_a$ is \textit{d}-reduced.
\end{proof}

\begin{lemma}\label{lem:join_of_clear_with_principal_closed}
 In a $\kappa$-frame $L$, $\D_L \vee \nabla_a = \partial_{(\downarrow a)^{**}}$.
\end{lemma}
\begin{proof}
 We find that $\cl(\D_L \vee \nabla_a) = \nabla_{\downarrow [a]_{\D_L}} = \nabla_{(\downarrow a)^{**}}$.
 But $\D_L \vee \nabla_a$ is clear by \cref{lem:principal_closed_in_clear_is_clear} and so the result follows.
\end{proof}

We noted in \cref{ex:clear_congs_incomparable} that the map $I \mapsto \partial_I$ is not monotone. Nonetheless
\cref{lem:join_of_clear_with_principal_closed} allows us to say something in a special case.
\begin{corollary}
 The map $a \mapsto \partial_{(\downarrow a)^{**}}$ is a $\kappa$-frame homomorphism from $L$ to $\uparrow \D_L \subseteq \C L$ with kernel $\D_L$.
 In particular, it is monotone and clear congruences of the form $\partial_{(\downarrow a)^{**}}$ are closed under finite meets and nonempty
 $\kappa$-joins.
\end{corollary}

\section*{Acknowledgements}

I would like to thank John Frith and Christopher Gilmour for discussions about the content and presentation of this paper.

\bibliographystyle{spmpsci}

\end{document}